\documentclass[11pt,leqno]{amsart}
\usepackage[legalpaper, margin=1in]{geometry}
\usepackage{graphicx}
\graphicspath{ {./images/} }
\usepackage{setspace}
\spacing{1}
\usepackage{mathptmx}
\usepackage{amssymb}
\usepackage[dvipsnames]{xcolor}
\usepackage{tabto}
\usepackage[scaled]{beramono}
\usepackage{listings}
\usepackage{appendix}
\usepackage{tikz}
\usepackage{caption}
\usepackage{blindtext}
\usepackage{enumitem}
\usepackage{lmodern}
\usepackage{multicol}
\usepackage{hyperref}
\hypersetup{
    colorlinks=true,
    linkcolor=blue,
    filecolor=magenta,      
    urlcolor=cyan,
    pdftitle={Overleaf Example},
    pdfpagemode=FullScreen,
    }

\usepackage{fancyhdr}

\usepackage{amsthm}
\theoremstyle{plain}
\newtheorem{theorem}{Theorem}[section]
\newtheorem{definition}[theorem]{Definition}
\newtheorem{example}[theorem]{Example}
\newtheorem{lemma}[theorem]{Lemma}

\newtheorem{corollary}[theorem]{Corollary}

\vspace{-10cm}
\usepackage{multicol}

\begin{document}
\title{Minimum Decomposition on Maxmin Trees} 
\author[Emmy Huang and Ray Tang]{Emmy Huang and Ray Tang}
\date{\today}
\address{Lexington High School, Lexington, MA 02421}
\email{24stu178@lexingtonma.org}
\address{Massachusetts Institute of Technology, Department of Computer Science, Cambridge, MA 02139}
\email{rayt724@mit.edu}
\maketitle


\begin{abstract}
    Maxmin trees are trees that consist of nodes that are either local minimums or maximums. Such trees were first studied by Postnikov. Later Dugan, Glennon, Gunnells, and Steingrimsson introduced the concept of weight to these trees and proved a bijection between maximum weight maxmin trees and permutations, defining weights for permutations. In addition, the q-Eulerian polynomial $E_n(x, q)$ is defined which relates descents and weights of permutations. This polynomial was later proven to exhibit a stabilization phenomenon by Agrawal et al. Extracting the formal power series $W_d(t)$ from the stabilization of these coefficients, $W_d(t)$ was conjectured to partially correspond to A256193. In our paper, we introduce a process called minimum decomposition to help us better understand maxmin trees. Using minimum decomposition, we present a new way to calculate the weight of different maxmin trees and prove the bijection between the coefficients of $W_d(t)$ and A256193.
\end{abstract}

\section{Introduction}

This paper studies how Maxmin Trees can be better understood through Minimum Decomposition. 

Tiered trees are trees with additional requirements on their labels. In studying tiered trees, Dugan, Glennon, Gunnells, and Steingrímsson \cite{2} defined the notion of weight for tiered trees which come up in the various geometric counting problems including the counting of certain representations of supernova quivers and torus orbits. These trees are generalizations of maxmin trees, which they study throughout their paper. Specifically, they prove a bijection between weight 0 maxmin trees and permutations. They also introduce multiple open questions about weights of permutations.

These have to do with the q-Eulerian Polynomial $E_n(x, q)$ defined in \cite{2} which extends the q-Eulerian Polynomial to include information about descents in the permutation. As conjectured in \cite{2} and proved in \cite{3}, the coefficients of these polynomials converge to a fixed sequence. Thus, a power series can be extracted, where the coefficients, as noted in \cite{2} correspond to the numbers in the OEIS sequence A256193. This sequence $T(n, k)$ counts the number of partitions of $n$ with $k$ parts of the second kind.

In our paper, we study maxmin trees and offer two central results concerning the q-Eulerian Polynomial and weights of permutations. Specifically, we prove the bijection between the coefficients of the power series and $T(n, k)$ conjectured in \cite{2}. We also give two new ways to find the weight of a permutation, including one iterative approach that can be implemented with lower complexity and another based on a new perspective to view maxmin trees. We include implementations of these algorithms in C++ on GitHub.

Thus, our paper is formatted as follows: We first introduce the general terminology of maxmin trees in section 2. We then consider weights of permutations in section 3, introducing the concept of minimum decomp trees. In section 4, we discuss the stabilization of $E_n(x,q)$ and use this information along with minimum decomp trees in section 5 to prove the bijection between $T(n, k)$ and the coefficients of $W_d(t)$.

\section{Definitions}

\begin{definition}
    \normalfont Let $S_n$ be the \textit{symmetric group} whose elements are all permutations of length $n$.
\end{definition}

Then $S_n$ consists of $n!$ different permutations of the set $[\![n]\!] = {1, \ldots, n}$. For example, the elements of $S_3$ are $\{1\ 2\ 3, 2\ 1\ 3, 1\ 3\ 2, 2\ 3\ 1, 3\ 1\ 2, 3\ 2\ 1\}$. A permutation $\sigma \in S_n$ is generally denoted by $a_1 a_2 a_3 \ldots a_n$.

\begin{definition}
    \normalfont A tree, $T$, is a graph of $n$ nodes and $n - 1$ edges with exactly $1$ path between every pair of nodes. Let $|T| = n$.
\end{definition}

\begin{definition}
    \normalfont A maxmin tree, $T$, is a tree of $n$ nodes with a labeling of each node and a well-ordering between the labels of each node. Furthermore, every node is either greater than each of its neighbors, a local max, or less than each of its neighbors, a local min.
\end{definition}

\begin{definition}
    \normalfont We define $T_i$ to be the subtree of $i$ in a maxmin tree, $T$, where $T_i = \{x \mid x \in T, x >= i,\forall j$ on the path between $x$ and $i$ in $T$, $j$ is in $T_i$\}
\end{definition}

\begin{definition}
    \normalfont We define des$(T)$ to be the number of local maximums in $T$.
\end{definition}

\begin{definition}
    \normalfont We define the \textit{weight of a maxmin tree} $T$ recursively:
\begin{enumerate}
    \item If $T$ consists of one node, $w(T)$ = 0.
    \item Otherwise, we let $m$ be the minimal node, and $T_1, \ldots, T_j$ be the trees created by deleting $m$ from $T$. Let $u_1, \ldots, u_j$ be the vertex in $T_1, \ldots, T_j$ respectively connected to $m$. Then we have:
    \begin{equation*}
        w(T) = \sum _{i=1}^{j}(d_i + w(T_i))
    \end{equation*}
where $d_i$ represents the number of vertices in $T_i$ that could be in place of $u_i$. Specifically, it represents the number of local maxes that are smaller than $u_i$. $d_i$ can also be referred to as the number of descents in $T_i$.
\end{enumerate}

\end{definition}

\begin{definition}
    \normalfont The \textit{weight of a permutation} $\sigma \in S_n$ is the weight of its maximum weight maxmin tree.
\end{definition}

To calculate the weight of a permutation $\sigma \in S_n$, we construct its maximum weight maxmin tree. It was proved in \cite{2} that there is a bijection between the permutations in $S_n$ and weight 0 maxmin trees with $n+1$ vertices. The paper describes an algorithm to construct 0 weight trees from a permutation. Similarly, to construct a maximum weight maxmin tree from a permutation $\sigma \in S_n$:
\begin{enumerate}
    \item We add the value $n+1$ to the end of the permutation $\tau$ to create a permutation $\tau \in S_{n+1}$ that corresponds to a maximum weight tree.
    \item Find the minimum value 1 in the permutation $\tau$. We split the permutation into $\tau_l \cdot 1 \cdot \tau_r$.
    \item For $\tau_l$, we start from the left and split it into $\tau_1 \cdot \tau_2$ where the rightmost element of $\tau_1$ is the global maximum of $\tau_l$. We continue this process recursively on $\tau_2$ and so on until $\tau_l = \tau_1 \cdot \tau_2 \cdots \tau_n$.
    \item We are left with the sequence $\tau_1 \cdot \tau_2 \cdots \tau_n \cdot 1 \cdot \tau_r$ where the rightmost element of each subtree is the global maximum. These should be connected the minimum element in the permutation
    \item Repeat this process for each subtree to find the maximum weight tree.
\end{enumerate}
We then use the weight formula to calculate the weight of this tree, and thus we find the weight of a permutation.

\section{Weights of Maxmin Trees and Permutations}

\begin{theorem}
    \normalfont The weight of a max-weight tree of a permutation is $w(T_m)=\sum_{x} \text{des}(T_x) - 1+w(T_x))$
\end{theorem}

\begin{proof}
By choosing the node with the largest label at each step, we add des$(T_x)-1$ to the weight at each step. 
\end{proof}

\begin{theorem}
\normalfont For a max weight tree $T$, if $m$ is the node with the smallest label in $T$:
    $$w(T_m) = \sum_{x} \text{des}(T_x) - \lvert T_m \rvert + 1$$
where $T_x$ enumerates all the trees connected to $T_m$.
\end{theorem}

\begin{proof}
We prove this with induction over the size of the tree. For size 1 trees, there are no trees connected to the single vertex, the size is 1, and the weight is 0, which validates the formula for the base case. Now, we assume the formula is correct for all trees of size $n$ and less. Using the definition of weight for a max weight tree, 
$$w(T_m)=\sum_{x} \text{des}(T_x) - 1+w(T_x)) = \sum_{x} \text{des}(T_x) - 1+w(T_x))$$.
\end{proof}

\begin{corollary}
    \normalfont The weight of a permutation is the sum of the number of descents in each of its subtrees $- \, n$.
\end{corollary}
We propose the following algorithm to find the weight of a permutation directly from the permutation.

\begin{theorem}
    \normalfont Let
    \begin{enumerate}
        \item $\sigma \in S_n$ (Note that we append $n+1$ and $0$ in the back and $n+2$ in the front to avoid bound errors).
        \item $i$ be a non-descent 
        \item $j$ be the first index to the right of $i$ such that $\sigma_j < \sigma_i$
        \item $m$ be the index $i < m < j$ such that for every $k \neq m$, $i < k < j, \sigma_m > \sigma_k$
        \item $M$ be the largest index $M < m$ such that for every $M < k < m, \sigma_M > \sigma_k$
        \item $L$ be the index $L < m$ such that for every $L < k < m$, $\sigma_k > \sigma_i$ and $\sigma_L < \sigma_i$
    \end{enumerate}
We claim that $$[\text{max}(M,L) + 1, m]$$ is the subtree of $i$.
\end{theorem}

\begin{proof}
    We start by proving the right side bound. Say that $r$ is the right side end of the subtree of an element $i$. We will prove that $r = m$.
    \begin{enumerate}
        \item If $r < m$, we let $k$ be where $\sigma_k < \sigma_i$ such that $k$ creates the split at $r$. By this, we mean that it splits the LHS into multiple trees by taking global maxes as described in the algorithm in Definition 7. Thus, we have $\sigma_k > \sigma_j$ since when taking $\sigma_j$ as in the algorithm, we did not create a split at $r$. Therefore, if $k$ creates the split at $r$, we must have $\sigma_r > \sigma_m$, which contradicts property 4 for $m$. Thus, $r \nless m$.
        \item If $r > m$, we have to check two cases. If $r > j$, $r$ could not reach $i$ because $j$ will create a split immediately to the right of index $m$. This is true because there are no elements between index $m$ and $j$ since $m$ is the global max by property 4. If $r < j$, because of property 4, we know $\sigma_m > \sigma_r$. Assume $k$ was the non-descent that created $r$ and $\sigma_k < \sigma_i$. Thus, when splitting the permutation for subtrees of $r$, we would have created a split immediately to the right of index $m$. Thus, since $r > m$, we cannot have $r$ in the subtree of $i$. Thus, $r \ngtr m$.
    \end{enumerate}
    Combined, we get $r = m$, so $m$, as described above, is the rightmost element of the subtree of $i$. Now, we prove the left side bound. Say that $l$ is the left side end of the subtree of $i$. We split it into two cases.
    \begin{enumerate}
        \item When $L < M$, we want to prove $l = m + 1$. There is guaranteed to be a cut at the right of $M$ and impossible to be a cut between $M$ and $m$, because all non-descents with values less than $\sigma_i$ are either to the left of $M$, to the right of $j$, or "blocked by $m$". As a result, there is a cut at $M$, and no cut between $M$ and $m$.
        \item When $L > M$, we want to prove $l = L + 1$. Again, every non-descent is either to the left of $L$, to the right of $j$, or "blocked by $m$".
    \end{enumerate}
    Thus, our theorem holds, and $[\max(M, L) + 1, m]$ is the subtree of $i$.
\end{proof}

As a result, with Corollary 3.3 we can apply this algorithm to each non-descent, find the sum of the number of descents in each range, and subtract n to get the weight of a permutation. 

We have implemented the algorithm to find the weight of permutation in Github in C++. \href{https://github.com/arwaeystoamneg/Tiered-Trees/blob/main/weight.cpp}{This code} runs in $O(n^2)$, where n is the size of the permutation; it is the same complexity the original algorithm runs in. However, one can use segments trees to pre-compute several of the variables, reducing the time to $O(n \cdot log(n))$. One can find the implementation \href{https://github.com/arwaeystoamneg/Tiered-Trees/blob/main/weight_nlogn.cpp}{here}.

\subsection{Minimum Decomposition of Maxmin Trees}

\hfill\\
Reconsider the algorithm to create the 0-weight maxmin tree of a permutation. 

\begin{definition}
    \normalfont We define \textit{the minimum decomp tree} of $\pi = \pi_1 \cdots \pi_L \cdot m \cdot \pi_R$, where $m$ is the minimum of $\pi$, as connecting $m$ to the minimum in each subtree $\pi_1, \cdots ,\pi_L, \pi_R$ instead of a descent. We will  view a minimum decomp tree as rooted at $m$ and denote the subtree of a node as all the nodes below it.
\end{definition}

\begin{example}
    \normalfont Take the permutation \\
    $\sigma = 1\ 12\ 15\ 9\ 10\ 5\ 7\ 11\ 6\ 4\ 13\ 3\ 8\ 2\ 14\ 16 \in S_{15}$ where 16 is appended to the end of the permutation as the extra node in the tree. \\
    After splitting once, we have:
    $1 \cdot 12\ 15\ 9\ 10\ 5\ 7\ 11\ 6\ 4\ 13\ 3\ 8\ 2\ 14\ 16$ so by our definition we have 1 is connected to 2.
    We then find the minimal element in the right permutation, 2, and again split the permutation with the maximum weight tree algorithm. Thus, we get: $12\ 15 \cdot 9\ 10\ 5\ 7\ 11\ 6\ 4\ 13 \cdot 3\ 8 \cdot 2 \cdot 14\ 16$ so by our definition, we get 2 is connected to 12, 4, 3, and 14.
    We continue this process, to get:
    \begin{enumerate}
        \item 12 is connected to 15
        \item 4 is connected to 5, 6, and 13
        \item 5 is connected to 9 and 7
        \item 9 is connected to 10
        \item 7 is connected to 11
        \item 3 is connected to 8
        \item 14 is connected to 16
    \end{enumerate}
    A visual representation of this tree is shown below:
    
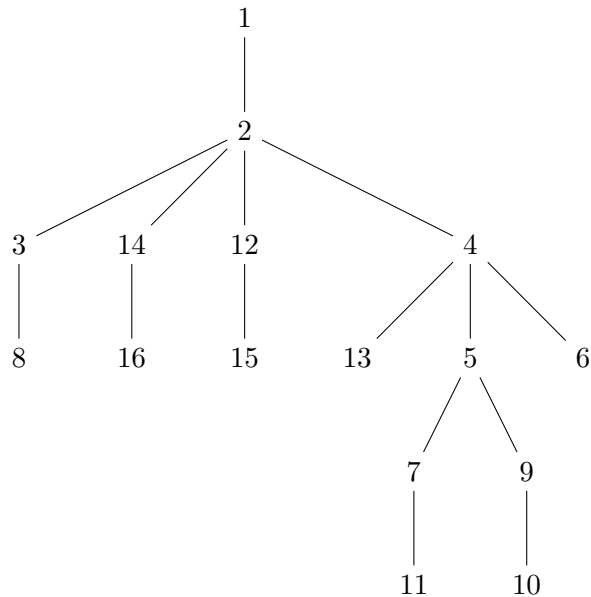
\begin{center}
    \begin{tikzpicture}
        \node {1} [sibling distance = 1.5cm]
            child {node {2}
                child {node {3}
                    child {node{8}}
                }
                child {node {14}
                    child {node {16}}
                }
                child {node {12}
                    child {node {15}}
                }
                child [missing]
                child {node {4}
                    child {node {13}}
                    child {node {5}
                        child {node {7}
                            child {node {11}}
                        }
                        child {node {9}
                            child {node {10}}
                        }
                    }
                    child {node {6}}
                }
            };
    \end{tikzpicture}
    \end{center}
\end{example}

\href{https://github.com/arwaeystoamneg/Tiered-Trees/blob/main/min_decomp.cpp}{Here} is a C++ implementation to find the minimum decomposition tree. \\
All this said, the minimum decomp tree has some interesting properties that we later use to prove a bijection between the coefficients of q-analouges of Eulerian polynomials and a certain type of partition denoted $T(n, k)$. Note: This will be defined later.

\subsection{Properties of the Minimum Decomp Tree}

We refer to the nodes of a minimum decomp tree that don't have any nodes below them as its stem. All nodes that are not leaves and the topmost node are part of its stem.
\begin{lemma}
    \normalfont Any valid minimum decomp tree corresponds to a unique permutation.
\end{lemma}

\begin{proof}
    Every minimum decomp tree can biject to its 0-weight counterpart by choosing the minimum possible descent at each node in the stem and using the algorithm in \cite{2}. Furthermore, it can be shown with induction that every minimum decomp tree corresponds to a unique 0-weight tree. When the sizes of the tree are $0$ or $1$, the result is trivial. Assuming that the result is true for all trees of size $1, \ldots, n$, we can prove the result for $n + 1$. We proceed with contradiction. Assume there exists two minimum decomp trees that correspond to the same 0-weight permutation. The top node of these trees must be the same and because of the inductive assumption, all the smaller trees connecting to the top node must also be the same. Furthermore, because the top node must connect to the minimums in all the smaller trees, we can conclude that the two trees are isomorphic to each other. 
\end{proof}

\begin{lemma}
    \normalfont For every index i, its subtree in a max weight tree are all the nodes below it in the minimum decomp tree.

\end{lemma}

\begin{proof}
    This is true by how the tree was constructed.
\end{proof}

\begin{lemma}
    \normalfont Every leaf of a minimum decomp tree is a descent and every descent is a leaf.
\end{lemma}

\begin{proof}
    Every descent is a leaf because having a larger than 1 sized subtree in the minimum decomp tree means it has a subtree in the original permutation. Every leaf have a subtree of size 1, and therefore, is a descent.
\end{proof}

\begin{corollary}
    \normalfont The weight of a permutation can be found by creating its minimum decomp tree, counting the number of leaves of each non-descent's subtree and subtracting n.
\end{corollary}

\begin{proof}
    This follows directly from Corollary 3.3.
\end{proof}

\begin{lemma}
    \normalfont When one "moves up" a leaf, i.e. the descent is attached to a node one higher on the stem, the weight of the tree decreases by 1.
\end{lemma}

\begin{proof}
    One non-descent loses a leaf in its subtree while all others remain unchanged. Thus, we just subtract one from the total weight.
\end{proof}

\section{The Stabilization of \texorpdfstring{$E_n(x, q)$}{Lg}}

When studying permutations, we often consider different permutation statistics, one of them being the number of descents in it. Thus, we can define a way to count these.

\begin{definition}
    \normalfont The \textit{Eulerian polynomial} is defined as:
    $$E_n(x) = \sum_{\sigma \in S_n} x^{\text{des}(\sigma)}$$
    where $\text{des}(\sigma)$ represents the number of descents in $\sigma$. The coefficients of these polynomials are called the \textit{Eulerian numbers}.
\end{definition}

\begin{example}
    \normalfont For example, we have $E_4(x) = 1 + 11x + 11x^2 + x^3$.
\end{example}

To include information about weights of permutations, we define an extension of the Eulerian polynomial with the addition of a variable $q$.

\begin{definition}
    \normalfont Define the \textit{q-Eulerian polynomial} as:
    $$E_n(x, q) = \sum_{\sigma \in S_n}x^{\text{des}(\sigma)} q^{w(\sigma)}$$
    where $\text{des}(\sigma)$ represents the number of descents in $\sigma$ and $w(\sigma)$ represents the weight of $\sigma$.
\end{definition}

\begin{example}
    \normalfont We give some of the first few q-Eulerian polynomials below: \\
    $E_3(x, q) = 1 + x(q + 3) + x^2$ \\
    $E_4(x, q) = 1 + x(q^2 + 3q + 7) + x^2(q^2 + 4q + 6) + x^3$ \\
    $E_5(x, q) = 1 + x(q^3 + 3q^2 + 7q + 15) + x^2(q^4 + 4q^3 + 11q^2 + 25q + 25)+ x^3(q^3 + 5q^2 + 10q + 10) + x^4$ \\
    $E_6(x, q) = 1 + x(q^4 + 3q^3 + 7q^2 + 15q + 31) + x^2(q^6 + 4q^5 + 11q^4 + 31q^3 + 58q^2 + 107q + 90)+ x^3(q^6 + 5q^5 + 16q^4 + 34q^3 + 76q^2 + 105q + 65)+ x^4(q^4 + 6q^3 + 15q^2 + 20q + 15) + x^5$
\end{example}

Fixing an integer $k$, we consider the coefficients $x^k$ in $E_n(x, q)$. As $n \rightarrow \infty$, the coefficients of $E_n(x, q)$ stabilize \footnote{stabilize meaning converge to a fixed sequence} as conjectured in \cite{2}.

\begin{example}
    The coefficients of $x^3$ from Example 5.2 seem to stabilize to
    $$q^M + 4q^{M-1} + 11q^{M-2} + 31q^{M-3} + ... $$
\end{example}

\begin{theorem}
    \normalfont For $k, d, m \in \mathbb{N}$ such that $m = d+k+1$ and $n \geq m$, the value of $E_n[x^dq^{\text{maxwt}(n,d)-k}]$ is stabilized.
\end{theorem}

This was proved in \cite{3} and we can thus extract the formal power series from the stabilization.

\begin{definition}
    \normalfont The \textit{power series $W_d(t) \in \mathbb{Z}[\![t]\!]$} for every $d \geq 1$ is defined as
    $$W_d(t) = 1 + a_1t + a_2t^2 + \cdots$$
    where the coefficients for $x^d$ as $n \rightarrow \infty$ stabilize to
    $$q^M + a_1q^{M-1} + a_2q^{M-2} + \cdots$$
    with $M = d(n-d-1)$ as the maximum weight for a permutation of length $n$ and $d$ descents (proved in \cite{2}).
\end{definition}

\begin{example}
    \normalfont The first few power series $W_d(t)$ are shown below:
    \begin{align*}
        W_1(t) &= 1 + 3t + 7t^2 + 15t^3 + 31t^4 + 63t^5 + \cdots \\
        W_2(t) &= 1 + 4t + 11t^2 + 31t^3 + 65t^4 + 157t^5 + \cdots \\
        W_3(t) &= 1 + 5t + 16t^2 + 41t^3 + 112t^4 + 244t^5 + \cdots \\
        W_4(t) &= 1 + 6t + 22t^2 + 63t^3 + 155t^4 + 393t^5 + \cdots \\
    \end{align*}
\end{example}

The coefficients of $W_d(t)$ was conjectured in \cite{2} to have a partial correspondece to the sequence A256193 in OEIS. This correspondence is proved in the next section.

\section{A Bijection between \texorpdfstring{$T(n, k)$}{Lg} and the Coefficients of \texorpdfstring{$W_d(t)$}{Lg}}

\subsection{\texorpdfstring{$T(n,k)$}{Lg}}
\begin{definition}
    \normalfont Define \textit{$T(n,k)$} as the number of partitions of $n$ into 2 types with exactly $k$ parts of the second type. 
\end{definition}

\begin{example}
    \normalfont Take $T(8, 5)$ for example. We first partition 8 into at least 5 parts and compute the number of different ways to assign 5 parts of the second sort.
    \begin{enumerate}
        \item 41111 : ${\binom{5}{5}} = 1$
        \item 32111 : ${\binom{5}{5}} = 1$
        \item 311111 : ${\binom{6}{5}} = 6$
        \item 22211 : ${\binom{5}{5}} = 1$
        \item 221111 : ${\binom{6}{5}} = 6$
        \item 2111111 : ${\binom{7}{5}} = 21$
        \item 11111111 : ${\binom{8}{5}} = 56$
    \end{enumerate}

  Adding, we get that $T(8, 5) = 92$. 
\end{example}

The table below shows the first few $T(n, k)$, $n \geq 0, 0 \leq k \leq n$:
\begin{table}[ht]
\caption{Coefficients of $T(n,k)$}
\centering
\begin{tabular}{c c c c c c c c c c c}
\textbf{1} & & & & & & & & & &\\
1 & \textbf{1} & & & & & & & & &\\
2 & \textbf{3} & \textbf{1} & & & & & & & &\\
3 & 6 & \textbf{4} & \textbf{1} & & & & & & &\\
5 & 12 & \textbf{11} & \textbf{5} & \textbf{1} & & & & &\\
7 & 20 & 24 & \textbf{16} & \textbf{6} & \textbf{1} & & & & &\\
11 & 35 & 49 & \textbf{41} & \textbf{22} & \textbf{7} & \textbf{1} & & & & \\
15 & 54 & 89 & 91 & \textbf{63} & \textbf{29} & \textbf{8} & \textbf{1} & & &\\
22 & 86 & 158 & 186 & \textbf{155} & \textbf{92} & \textbf{37} & \textbf{9} & \textbf{1} & &\\
30 & 128 & 262 & 351 & 342 & \textbf{247} & \textbf{129} & \textbf{46} & \textbf{10} & \textbf{1} &\\
42 & 192 & 428 & 635 & 700 & \textbf{590} & \textbf{376} & \textbf{175} & \textbf{56} & \textbf{11} & \textbf{1}
\end{tabular}
\caption*{The numbers in bold correspond to coefficients in $W_d(t)$.}
\label{Tab:Tcr}
\end{table}

\subsection{Bijection to \texorpdfstring{$W_d(t)$}{Lg}}

We will denote node $n+1$ as a descent.
\begin{theorem}
\normalfont Consider all permutations of length $n$ with $d$ descents and $(n-d-1)(d-1)$ weight. These correspond to trees of $n+1$ nodes, $d+1$ descents\footnote{We consider the added node $n+1$ a descent}, and $(n-d-1)(d-1)$ weight. The number of these permutations is equal to the ($n-d$)th coefficient in $W_d(t)$.
We claim there is a bijection between this coefficient and $T(n-1, d)$ when $n \geq 2d$.
\end{theorem}

Note that since the coefficients of the q-Eulerian polynomials stabilize, as mentioned above they correspond to certain coefficients in the power series, we only have to prove the base case for each $T(n, k)$.

To prove this theorem, we first need the following lemmas:

\begin{lemma}
\normalfont The maximum possible weight of a tree with $d + 1$ descents and $n - d$ non-descents is $(n - d - 1) \cdot d$.
\end{lemma}

\begin{proof}
Using minimum decomposition, we know the maximum weight when there are d + 1 leaves and n - d non-descents is achieved when all the leaves are below all the non-descents i.e. the non-descents form a straight line and the leaves are all connected to the bottom node. From Theorem 4.7, this weight is equal to:
$$(n - d)(d + 1) - n = nd - d^2 - d = d(n - d - 1)$$
Thus, we have found the maximum possible weight of this tree.
\end{proof}

\begin{lemma}
\normalfont The minimum decomp trees for trees with $(n-d-1)(d-1)$ weight will always be a stem of a straight line and all the leaves connected to various nodes on the stem.
\end{lemma}

\begin{proof}
Consider the maximum weight tree of a permutation with $d$ descents and $n-d$ non-descents. The difference between this and our desired weight is $n-d-1$.
When $2d \geq n$, we have $n - d - 1 \leq d - 1 < d + 1$.

If node of the stem branch from the straight line, i.e. we have the following structure ($x_a$ and $x_b$ are non-descents and part of the stem):
    
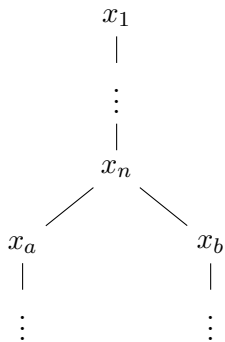
\begin{center}
    \begin{tikzpicture}
        \node {$x_1$} [sibling distance = 2.5cm] [level distance = 1cm]
            child {node {$\vdots$}
                child {node {$x_n$}
                    child{node {$x_a$}
                        child {node {$\vdots$}}
                    }
                    child {node {$x_b$}
                        child {node {$\vdots$}}
                    }
                }
            };
    \end{tikzpicture}
    \end{center}
the weight will be reduced by at least $d+1$ because there will be a loss of at least $d+1$ stem-leaf pairs. As a result, it is impossible for the stem to branch or the weight will drop too fast.
\end{proof}

Now, we can prove the earlier theorem.
We will approach this by looking at the minimum decomposition of the trees. 
Since the stem is a straight line and the weight is $n-d-1$ too high, we can get all desired minimum decomp trees by starting with the heaviest minimum decomp tree and moving the leaves up $n-d-1$ times. 

Lets approach this problem by looking at different cases for the number of leaves connected to each stem, and for each case, the values of the stem.
\begin{example}
    \normalfont Lets look at the case when $n = 9$ and $d = 5$. Since $n - d = 4$, there will be 4 elements in the stem. We will also push up 3 leaves from the max weight case as $n-d-1 = 3$. As a result, our stem can have structure 1 0 0 5, 0 1 1 4, or 0 0 3 4 where each number represents the number of leaves connected to each stem node to reach the desired weight. 
    
    When we have 1 0 0 5, we have several possibilities of the nodes on the stem. We can have:
    \begin{itemize}
        \item 1 2 3 4: $\binom{6}{1}$ = 6 choices since we pick which descent connects to 1
        \item 1 2 3 5: 1 choice since we need to put descent 4 on 1 and everything else on 5
        \item 1 2 4 5: 1 choice with the same reasoning as above
        \item 1 3 4 5: 1 choice
    \end{itemize}
    When we have 0 1 1 4, we could have:
    \begin{itemize}
        \item 1 2 3 4: $\binom{6}{1}$ $\binom{5}{1}$ = 30 choices since we pick which two descents to connect to 2 and 3
        \item 1 2 3 5: $\binom{2}{1}$ $\binom{5}{1}$ = 10 choices since descent 4 connects with either 2 and 3 and we must pick which other descent connects to the other one of 2 and 3
        \item 1 2 3 6: $\binom{2}{1}$ = 2 choices since we pick whether descent 4 or 5 connects to 2 or 3
        \item 1 2 4 5: $\binom{5}{1}$ = 5 choices since descent 3 connects to 2 and we pick another descent which connects to 4
        \item 1, 2, 4, 6 : 1 choice
    \end{itemize}
    Finally, for 0 0 3 4, we could have:
    \begin{itemize}
        \item 1 2 3 4: $\binom{6}{3}$ = 20 choices
        \item 1 2 3 5: $\binom{5}{2}$ = 10 choices
        \item 1 2 3 6: $\binom{4}{1}$ = 4 choices
        \item 1 2 3 7: 1 choice
    \end{itemize}
    Added together, we get a total of 9 + 48 + 35 = 92 which is equal to $T(8, 5)$.
\end{example}

The values grouped by the stem-leaf values do not appear to have a pattern. However, grouping by the elements in the stem reveals an interesting correspondence. We will instead look at the number of trees corresponding to each stem:

\begin{example}
    \normalfont We group the items 
    \begin{itemize}
        \item 1 2 3 4: 6 + 30 + 20 = 56
        \item 1 2 3 5: 1 + 10 + 10 = 21
        \item 1 2 3 6: 2 + 4 = 6
        \item 1 2 3 7: 1
        \item 1 2 4 5: 1 + 5 = 6
        \item 1 2 4 6: 1
        \item 1 3 4 5: 1
    \end{itemize}
    Comparing this with example 6.2, these appear to be the same numbers as the different partition cases.
\end{example}

Observing from above, it seems the sum of the values in the stem of the tree are equal, the number of trees they can correspond to are also equal.

\begin{theorem}
    \normalfont The number of minimum decomp trees of $n+1$ nodes and $d+1$ descents corresponding to the stem:
    $$x_1 \cdot x_2 \cdot x_3 \cdots x_{n-d}$$
    and the number of ways to split the partition of $n-1$:
    $$x_{n-d}-(n-d-1),x_{n-d-1}-(n-d-2), \cdots ,x_2 - 1,x_1,1,1,\cdots$$
    where there are:
    $$n - 1 - \sum_{i = 1}^{n - d}x_i + \sum_{i=1}^{n-d-1}i$$
    1's into two types where there are $d$ parts of the second type are equal.
\end{theorem}

\begin{proof}
    For the listed partition, there are:
    $${\binom{n - 1 - (\sum_{i = 1}^{n - d}x_i) + (\sum_{i=1}^{n-d-1}i) + n - d}{d}} = {\binom{n - 1 - \sum_{i = 1}^{n - d}(x_i+i)}{d}}$$
    different ways to assign $d$ values of the second kind.
    
    Now, we count the number of trees corresponding to $x_1 \cdot x_2 \cdot x_3 \cdots x_{n-d}$. Note that $x_1 = 1$ or 1 would need to be a descent so we substitute this in for the rest of the proof. 
    
    Furthermore, there are at least $x_2 - x_1 - 1$ nodes that are connected to $x_1$, $x_3-x_2-1$ additional nodes connected to $x_2$ or $x_1$, $x_3-x_2-1$ additional nodes that are connected to $x_3$, $x_2$, or $x_1$ and so on (this is because all the values between $x_1$ and $x_2$ have to be connected to $x_1$, all the values between $x_2$ and $x_3$ have to be connected to $x_2$ or $x_1$, and so forth).
    
    We can count the number of trees by first attaching these leaves the highest value stem node that they can be on (i.e. $x_2 - x_1 - 1$ to $x_1$, $x_3 - x_2 - 1$ to $x_2$, and so on) and moving up the leaves however many more times to get the desired weight. It should be noted that we can move any node up as many times as we want because it is impossible to have enough weight to ``move above 1". 
    Thus, the weight disparity, i.e. the difference between the current weight and the maximum weight, is:
    $$1(x_{n - d} - x_{n - d - 1} - 1) + 2(x_{n - d - 1} - x_{n-d-2})+ \cdots$$
    $$+(n-d-1)(x_2 - x_1-1) + (n - d)(x_1-1) = \sum_{i=1}^{n-d}(x_i-i)
    $$
    
    As a result, we need to decrease the weight or ``move up" the leaves by
    $$n - d - 1 - \sum_{i=1}^{n-d}(x_i)+\sum_{i=1}^{n-d}(i)$$
    more. We have to distribute these ``move ups" amongst $d+1$ leaves, and so by stars and bars, the total becomes:
    $${\binom{n - d - 1 - (\sum_{i=1}^{n-d}x_i) + (\sum_{i=1}^{n-d}i + d)}{d}} = {\binom{n - 1 - \sum_{i=1}^{n-d}(x_i+i)}{d}}$$
    
    Thus, the theorem is true.
\end{proof}

The above approach can be used for cases when $n > 2d$. However, as $n/d$ grows larger, the different stem structures grows increasingly complex and does not follow the same pattern. When $n < 3d$ there is still a lot we can say about the structure of the stem (ex: there can only be 1 off-branching node), however, larger $n/d$ values make this very difficult. One may try to find a way to combine data of $T(n, k)$ to find further coefficients of the power series.

\section{Acknowledgements}

We would like to thank the PROMYS program and the Clay Institute for giving us the opportunity to work on this project. We would also like to thank our mentors Paul Gunnells, David Fried, Hana Ephremidze, and John Sim for all their support and encouragement throughout the process.

\end{document}